\documentclass{amsart}
\usepackage{graphicx}
\usepackage{calc}

\usepackage[cmtip,arrow]{xy}
\usepackage{pb-diagram,pb-xy}
\usepackage{amsmath,amsthm,amssymb,tabularx}

\allowdisplaybreaks

\newcommand{\abs}[1]{\left\vert#1\right\vert}

\newtheorem{theorem}{Theorem}

\def\supp{\operatorname{\sf supp}}
\def\id{\operatorname{\sf id}}

\begin{document}

\title{Stereotype approximation property for the group algebra ${\mathcal C}^\star(G)$ of measures}

\author{S.~S.~Akbarov}

\address{All-Russian Institute of Scientific and Technical Information \\ (VINITI RAN) \\
 Usievicha 20,\\ Moscow, A-190, 125190, Russia}

\email{sergei.akbarov@gmail.com}

\thanks{Supported by the RFBR grant No. 18-01-00398.}

\maketitle

In \cite{Akbarov-approx,Akbarov} the author described the stereotype approximation property, an analog of the classical approximation property transferred into the category ${\tt Ste}$ of stereotype spaces. It was noticed in \cite{Akbarov} that for a stereotype space $X$ the stereotype approximation is formally a stronger condition than the classical approximation (although up to now it remains unclear whether these conditions are equivalent or not). For this reason the question of which concrete spaces in the standard package used in functional analysis have the stereotype approximation property is quite difficult (the only exception is the case when the space has a topological basis in some reasonable sence). In this paper we prove the following fact concerning the stereotype group algebras $\mathcal C^\star(G)$ (defined in \cite{Akbarov-group-algebras,Akbarov}).

\begin{theorem} For any locally compact group $G$ the stereotype spaces $\mathcal
C(G)$ of continuous functions and $\mathcal C^\star (G)$ of measures with compact support on $G$ have the stereotype approximation property. \end{theorem}
\begin{proof}
The stereotype approximation is inherited under the passage to the stereotype dual space \cite[Theorem 9.5]{Akbarov}, hence it is sufficient to consider the space $\mathcal C(G)$. Each locally compact group $G$ has a $\sigma$-compact open subgroup $H$, and for such a subgroup the space $\mathcal C(G)$ is a direct sum of some family of copies of the space $\mathcal C(H)$:
$$
\mathcal C(G)\cong \prod_{g\cdot H\in G/H} \mathcal C(g\cdot H)
$$
This implies that we can consider only the case when the group $G$ is $\sigma$-compact. Let us fix a compact neighbourhood of unity $U_0$ in $G$. For any compact neighbourhood of unity $U\subseteq U_0$ let us choose a function $\eta_U \in \mathcal C(G)$ with the following properties (here $\mu$ is a given left-invariant Haar measure):
$$
\supp \eta_U\subseteq U, \quad \eta_U \ge 0, \quad \int_G \eta_U(t)\cdot
\mu(dt)=1.
$$
Then $\eta_U$ is a neighbourhood tending to the $\delta$-function:
\begin{multline*}
\eta_U * f (t)= \int_G \eta_U (s)\cdot f(s^{-1}\cdot t) \cdot \mu(ds)=\\=
\int_G \Delta\big(s^{-1}\big)\cdot \eta_U (t\cdot s^{-1})\cdot f(s) \cdot \mu(ds) \quad
\underset{U\to 1}{\longrightarrow} \quad f(t)
\end{multline*}
($\Delta$ is the modular function on $G$, cf. \cite{Hewitt-Ross}).

Let us show that the net of operators (not of finite rank so far)
$$
P^U : \mathcal C(G) \to \mathcal C(G) \quad | \quad P^U f=\eta_U * f
$$
tends to the identity in $\mathcal L(\mathcal C(G))$:
\begin{equation}\label{P^U->1}
P^U \, \underset{U\to 1}{\overset{\mathcal L(\mathcal C(G))}{\longrightarrow}}
\, \id_{\mathcal C(G)}
\end{equation}

1. By the Ascoli theorem \cite[8.2.10]{Engelking} each compact set $\Phi \subseteq \mathcal C(G)$ is a set of functions equicontinuous on each compact set $K\subseteq G$. Hence for given $K\subseteq G$ and $\varepsilon>0$ there is a neighbourhood of unity $V$ in $G$ such that
$$
\sup_{f\in\Phi,\ t\in K,\ s\in V}\abs{f(s^{-1}\cdot t)-f(t)}<\varepsilon.
$$
For any neighbourhood $U\subseteq V$, for any function $f\in\Phi$ and for any point $t\in K$ we have
\begin{multline*}
\abs{(P^Uf-f)(t)}=\abs{(\eta_U*f-f)(t)}=\abs{\int_G \eta_U (s)\cdot f(s^{-1}\cdot t) \cdot \mu(ds)-f(t)}=\\=
\abs{\int_G \eta_U (s)\cdot f(s^{-1}\cdot t) \cdot \mu(ds)-f(t)\cdot \int_G \eta_U (s)\ \mu(d s)}=\\=
\abs{\int_G \eta_U (s)\cdot \Big( f(s^{-1}\cdot t) -f(t)\Big)\ \mu(d s)}\le\\ \le
\underbrace{\max_{t\in K, \ s\in U}\abs{f(s^{-1}\cdot t) -f(t)}}_{\scriptsize\begin{matrix}\text{\rotatebox{90}{$>$}}\\ \varepsilon\end{matrix}}\cdot \underbrace{\int_G \eta_U (s)\ \mu(d s)}_{\scriptsize\begin{matrix}\|\\ 1\end{matrix}}<\varepsilon.
\end{multline*}
We can conclude that the operators $P^U$ tend to $\id_{\mathcal C(G)}$ uniformly on compact sets $\Phi \subseteq \mathcal C(G)$:
$$
P^Uf=\eta_U * f \, \underset{U\to 1, f\in \Phi}{\overset{\mathcal
C(G)}{\rightrightarrows}} \, f,
$$
In other words, the net $\{ P^U ; U\subseteq U_0\}$ tends to the identity operator in the space ${\mathcal C}(G):{\mathcal C}(G)$ (which was defined in \cite{Akbarov} as the space of operators $\varphi:{\mathcal C}(G)\to {\mathcal C}(G)$ with the topology of uniform convergence on compact sets in ${\mathcal C}(G)$):
$$
P^U \, \underset{U\to 1}{\overset{\mathcal C(G) : \mathcal
C(G)}{\longrightarrow}} \, \id_{\mathcal C(G)}
$$

2. Let us show that the net $\{ P^U ; U\subseteq U_0\}$ is totally bounded in ${\mathcal C}(G):{\mathcal C}(G)$. Take a function $f\in \mathcal C(G)$ and note that the net of functions $\{ P^Uf; U\subseteq U_0\}$ is uniformly bounded on each compact set $K\subseteq G$:
\begin{multline*}
\sup_{U\subseteq U_0}\abs{P^Uf(t)}=\sup_{U\subseteq U_0}\abs{(\eta_U*f)(t)}=
\sup_{U\subseteq U_0}\abs{\int_G \eta_U (s)\cdot f(s^{-1}\cdot t) \cdot \mu(ds)}\le\\ \le
\max_{t\in K, \ s\in U_0}\abs{f(s^{-1}\cdot t) -f(t)}\cdot \sup_{U\subseteq U_0}\underbrace{\int_G \eta_U (s)\ \mu(d s)}_{\scriptsize\begin{matrix}\|\\ 1\end{matrix}}=\\=\max_{t\in K, \ s\in U_0}\abs{f(s^{-1}\cdot t) -f(t)}<\infty.
\end{multline*}
On the other hand, the net $\{ P^Uf; U\subseteq U_0\}$ is equicontinuous on each compact set $K\subseteq G$, since for any $\varepsilon>0$ we can find a neighbourhood of unity $V$ such that
$$
\Big(t_1,t_2\in U_0^{-1}\cdot K\ \&\ t_1^{-1}\cdot t_2\in V\Big)\quad\Rightarrow\quad \abs{f(t_1)-f(t_2)}<\varepsilon.
$$
and then for $t_1,t_2\in K$ with the condition $t_1^{-1}\cdot t_2\in V$ and for $s\in U\subseteq U_0$ we have
$$
(\underbrace{s^{-1}\cdot t_1}_{\scriptsize\begin{matrix}\text{\rotatebox{90}{$\owns$}}\\ U_0^{-1}\cdot K\end{matrix}})^{-1}\cdot (\underbrace{s^{-1}\cdot t_2}_{\scriptsize\begin{matrix}\text{\rotatebox{90}{$\owns$}}\\ U_0^{-1}\cdot K\end{matrix}})=t_1^{-1}\cdot s\cdot s^{-1}\cdot t_2=t_1^{-1}\cdot t_2\in V
\quad\Rightarrow\quad \abs{f(s^{-1}\cdot t_1)-f(s^{-1}\cdot t_2)}<\varepsilon,
$$
hence
\begin{multline*}
\abs{P^Uf(t_1)-P^Uf(t_2)}=\abs{(\eta_U*f)(t_1)-(\eta_U*f)(t_2)}=\\
=\abs{\int_G \eta_U (s)\cdot\Big( f(s^{-1}\cdot t_1)-f(s^{-1}\cdot t_2)\Big) \cdot \mu(ds)}\le\\ \le
\underbrace{\max_{s\in U_0}\abs{f(s^{-1}\cdot t_1)-f(s^{-1}\cdot t_2)}}_{\scriptsize\begin{matrix}\text{\rotatebox{90}{$>$}}\\ \varepsilon\end{matrix}}\cdot \underbrace{\int_G \eta_U (s)\ \mu(d s)}_{\scriptsize\begin{matrix}\|\\ 1\end{matrix}}<\varepsilon
\end{multline*}
Thus, the net $\{ P^Uf; U\subseteq U_0\}$ is uniformly bounded and equicontinuous on each compact set $K\subseteq G$. By the Ascoli theorem \cite[8.2.10]{Engelking}, $\{ P^Uf; U\subseteq U_0\}$ is totally bounded in the space $\mathcal C(G)$. This is true for any function $f\in \mathcal C(G)$, therefore by \cite[Theorem 5.1]{Akbarov}, the net $\{ P^U ; U\subseteq U_0\}$ is totally bounded in the space ${\mathcal C}(G):{\mathcal C}(G)$.

Let us summarize what we told in 1 and 2: the net $\{ P^U ; U\subseteq U_0\}$ is totally bounded and tends to the identity operator in the space ${\mathcal C}(G):{\mathcal C}(G)$. This means that it tends to the identity operator in the pseudosaturation $\mathcal L(\mathcal C(G))=\Big({\mathcal C}(G):{\mathcal C}(G)\Big)^\vartriangle$ of the space ${\mathcal C}(G):{\mathcal C}(G)$. I.e. \eqref{P^U->1} holds.

Now it is sufficient to show that each operator $P^U$ can be approximated in
$\mathcal L(\mathcal C(G))$ by some operators of finite rank. Indeed, using the  $\sigma$-compactness of $G$ we can choose a {\it sequence} of functions $F_n$ in $\mathcal C(G\times G)$ such that
\begin{itemize}
 \item[(i)] $F_n (t,s) \, \underset{n\to
\infty}{\overset{\mathcal C(G\times G)}{\longrightarrow}} \, \Delta(s^{-1})\cdot\eta_U(t\cdot
s^{-1})$
 \item[(ii)] $F_n (t,s)=\sum_{k=1}^N u_k (t)\cdot
v_k(s)$, where $N\in{\mathbb N}$ depends on $n$, and $u_k,v_k\in \mathcal C(G)$ are functions with compact support on $G$.
\end{itemize}
Then the operators of finite rank
$$
P^U_n f(t)=\int_G F_n(t,s)\cdot f(s)\  \mu(d s)=\sum_{k=1}^N u_k(t) \cdot \int_G v_k(s) \cdot f(s) \cdot \mu (ds)
$$
tend to $P^U$ in $\mathcal C(G) : \mathcal C(G)$. Since $P^U_n$ is a converging {\it sequence}, it is totally bounded in ${\mathcal C}(G):{\mathcal C}(G)$. Therefore $P^U_n$ tends to $P^U$ in $\Big({\mathcal C}(G):{\mathcal C}(G)\Big)^\vartriangle=\mathcal L(\mathcal C(G))$.
\end{proof}

\end{document}